\documentclass[12pt,a4paper]{amsart}
\usepackage{amssymb}

\usepackage{graphicx}      
\usepackage[abs]{overpic}
\usepackage[usenames]{color}
\usepackage[varg]{newtxmath}\usepackage{newtxtext}
\usepackage[pdftex,bookmarks=true,bookmarksnumbered=true,bookmarkstype=toc,colorlinks=true,linkcolor=blue,citecolor=blue,setpagesize=false]{hyperref}

\newtheorem{thm}{Theorem}
\newtheorem{lem}[thm]{Lemma}
\newtheorem{cor}[thm]{Corollary}

\newtheorem{knownthm}{Theorem} 

\theoremstyle{definition}
\newtheorem{defn}[thm]{Definition}

\theoremstyle{remark}
\newtheorem{rem}[thm]{Remark}

\numberwithin{equation}{section}
\numberwithin{thm}{section}

\newcommand{\refeq}[1]{\textup{(\ref{eq:#1})}}
\newcommand{\refthm}[1]{Theorem \ref{thm:#1}}
\newcommand{\reflem}[1]{Lemma \ref{lem:#1}}
\newcommand{\refcor}[1]{Corollary \ref{cor:#1}}

\newcommand{\refdefn}[1]{Definition \ref{defn:#1}}

\newcommand{\reffig}[1]{Figure \ref{fig:#1}}

\newcommand{\const}{C}
\newcounter{const}
\newcommand{\cnst}{\refstepcounter{const}\const_{\theconst}}
\newcommand{\clabel}[1]{\cnst\label{#1}}
\newcommand{\refc}[1]{\const_{\ref{c:#1}}} 

\newcommand{\ve}{\varepsilon}
\newcommand{\bd}{\partial}          
\newcommand{\cl}[1]{\overline{#1}}  
\newcommand{\clB}{\cl{B}}           
\newcommand{\BB}{{B^*}}       

\DeclareMathOperator{\dist}{dist}

\DeclareMathOperator{\supp}{supp}

\newcommand{\qtext}{\quad\text}
\newcommand{\sm}{\setminus}
\newcommand{\ds}{\displaystyle}

\newcommand{\mass}[1]{\|#1\|}

\newcommand{\normt}[1]{\|#1\|_2}
\newcommand{\normi}[1]{\|#1\|_\infty}

\newcommand{\increase}{\uparrow}

\newcommand{\inv}{^{-1}}            

\makeatletter
\newcommand{\union}{\mathop{
\if@display
\bigcup\else\operatorname{\hbox{\small$\bigcup$}}
\fi}}
\newcommand{\barint}{\mathop{
\if@display
 \hbox{\vrule height3.2pt depth-2.5pt width.65em}\hskip-1em\int\hskip-.4em
\else
 \hbox{\vrule height3.5pt depth-3.1pt width.6em}\hskip-.85em\int\hskip-.1em
\fi}\ilimits@}

\newcommand{\pder}[2]{\if@display\dfrac{\partial#1}{\partial#2}
\else\partial#1/\partial#2\fi}

\newcommand{\redname}{\widehat{\mathbf{R}}}  
\def\red{\@ifnextchar[{\redDEu}{\redEu}}
\def\redDEu[#1]#2#3{{\null}^{#1}\redname^{#2}_{#3}}
\def\redEu#1#2{\redname^{#1}_{#2}}

\newcommand{\Redname}{\mathbf{R}}  
\def\Red{\@ifnextchar[{\RedDEu}{\RedEu}}
\def\RedDEu[#1]#2#3{{\null}^{#1}\Redname^{#2}_{#3}}
\def\RedEu#1#2{\Redname^{#1}_{#2}}

\def\hmname{\omega}
\def\hm{\@ifnextchar[{\hmat}{\hmm}}
\def\hmm#1#2{\hmname(#1, #2)}            
\def\hmat[#1]#2#3{\hmname^{#1}(#2,#3)}  

\def\hmdom{\@ifnextchar[{\hmatdom}{\hmmdom}}
\def\hmmdom#1{\hmname(#1)}            
\def\hmatdom[#1]#2{\hmname^{#1}(#2)}  

\makeatother

\newcommand{\di}{n}			    
\newcommand{\dom}{D}       
\newcommand{\bdy}{{\bd\dom}}       
\newcommand{\Dom}{\Omega}       

\newcommand{\bots}{\lambda_{\text{min}}}  

\newcommand{\feigen}{\lambda_\dom} 
\newcommand{\feigenf}{\varphi_\dom} 

\newcommand{\HH}{\mathbb H}
\newcommand{\Z}{\mathbb Z}

\newcommand{\EE}{\widetilde E}
\newcommand{\ddom}{\widetilde \dom}

\newcommand{\Const}{\const}

\newcommand{\cw}{w_\eta}  

\DeclareMathOperator{\capa}{Cap}
\newcommand{\capain}[1]{\capa_{#1}}

\DeclareMathOperator{\Ric}{Ric}

\newcommand{\Czi}{C_0^\infty}
\newcommand{\Gdo}{G_\dom^o}

\newcommand{\inmet}{d_\dom}
\newcommand{\dis}{\delta_\dom}

\newcommand{\koseimarku}[1]{\rlap{\vrule\raise .8em\hbox{\smash{\underbar{\bf\fontsize{3pt}{0pt}\selectfont #1}}}}}

\newcommand{\Corrdel}[1]{{\color{magenta}\koseimarku{del}\relax}{\color{black}}}
\newcommand{\Corr}[1]{{\color{magenta}#1}{\color{black}}}

\newcommand{\mama}{\relax}
\newcommand{\gobble}[1]{\relax}

\newcommand{\Draft}{\let\corr\Corr\let\corrdel\Corrdel}
\newcommand{\Final}{\let\corr\mama\let\corrdel\gobble}

\Draft

\begin{document}

\title[IU for manifolds]{Intrinsic Ultracontractivity for domains in negatively curved manifolds}

\dedicatory{Dedicated to the memory of Professor Walter K. Hayman}

\author{Hiroaki Aikawa}
\address{
College of Engineering,
Chubu University,
Kasugai 487-8501, Japan
}
\email{aikawa@isc.chubu.ac.jp}

\author{Michiel van den Berg}
\address{
School of Mathematics, University of Bristol, Fry Building, Woodland Road, Bristol BS8 1UG, United Kingdom
}
\email{mamvdb@bristol.ac.uk}

\author{Jun Masamune}
\address{
Department of Mathematics,
Hokkaido University,
Sapporo 060-0810, Japan
}
\email{jmasamune@math.sci.hokudai.ac.jp}


\thanks{This work was supported by JSPS KAKENHI Grant Number 17H01092.
MvdB was also supported by The Leverhulme Trust through Emeritus Fellowship EM-2018-011-9.}

\date{March 17, 2021}

\begin{abstract}
Let $M$ be a complete, non-compact, connected Riemannian manifold
with Ricci curvature bounded from below by a negative constant.
A sufficient condition is obtained for open and connected sets $D$ in $M$
for which the corresponding Dirichlet heat semigroup is intrinsically  ultracontractive.
That condition is formulated in terms of capacitary width.
It is shown that both the reciprocal of the bottom of the spectrum of the Dirichlet Laplacian acting in $L^2(D)$,
and the supremum of the torsion function for $D$
are comparable with the square of the capacitary width for $D$
if the latter is sufficiently small.
The technical key ingredients are the volume doubling property,
the Poincar\'e inequality
and the Li-Yau Gaussian estimate for the Dirichlet heat kernel for finite scale.
\end{abstract}

\maketitle

\section{Main results}\label{sec:main}

Let $M$ be a complete, non-compact, $\di$-dimensional connected Riemannian manifold,
without boundary, and with  Ricci curvature bounded below by a negative constant,
i.e., $\Ric\ge -K$ with nonnegative constant $K$.
Throughout the paper, $K$ is reserved for this constant.
In this article, we investigate domains (open, and connected sets) in $M$
for which the heat semigroup is intrinsically ultracontractive.

For a domain $\dom\subset M$ we denote by
$p_\dom(t,x,y)$, $t>0$, $x,y\in\dom$,
the Dirichlet heat kernel for $\pder{}t-\Delta$ in $\dom$,
i.e., the fundamental solution to $(\pder{}t-\Delta)u=0$
subject to the Dirichlet boundary condition
$u(t,x)=0$ for $x\in\bdy$ and $t>0$.
Davies and Simon \cite{MR766493} introduced the notion of
intrinsic ultracontractivity.
There are several equivalent definitions for intrinsic ultracontractivity
(\cite[p.345]{MR766493}).
The following is in terms of the heat kernel estimate.

\begin{defn}\label{defn:IU}
Let $\dom\subset M$.
We say that the semigroup associated with $p_\dom(t,x,y)$ is
\emph{intrinsically ultracontractive} (abbreviated to IU)
if the following two conditions are satisfied:
\begin{enumerate}
\item
The Dirichlet Laplacian $-\Delta$ has no essential spectrum and
has the first eigenvalue $\feigen >0$ with
corresponding positive eigenfunction $\feigenf$
normalized by $\normt{\feigenf}=1$.

\item
For every $t>0$, there exist constants $0<c_t<C_t$
depending on $t$ such that
\begin{equation}\label{eq:IU}
c_t\feigenf(x)\feigenf(y)
\le p_\dom(t,x,y)
\le C_t\feigenf(x)\feigenf(y)
\qtext{for all }x,y\in\dom.
\end{equation}
\end{enumerate}
For simplicity, we say that $\dom$ itself is IU
if the semigroup associated with $p_\dom(t,x,y)$ is IU.
\end{defn}

Both the analytic and probabilistic aspects of IU have been investigated in detail.
For example it turns out that IU implies the Cranston-McConnell inequality,
while IU is derived from very weak regularity of the domain.
Davis \cite{MR1124297} showed that a bounded Euclidean domain above the graph
of an \emph{upper semi-continuous} function is IU; no regularity of the boundary function is needed.
There are many results on IU for Euclidean domains.
Ba{\~n}uelos and Davis \cite[Theorems 1 and 2]{MR1206335} gave conditions
characterizing IU and the Cranston-McConnell inequality when restricting to a certain class of plane domains, which illustrate subtle difference between IU and the Cranston-McConnell inequality.
M{\'e}ndez-Hern{\'a}ndez \cite{MR1755257} gave further extensions.
See also
\cite{MR3420485},
\cite{MR1124298},
\cite{MR1151804},
\cite{MR1269281},
\cite{MR1124297},
and references therein.

There are relatively few results for domains in a Riemannian manifold.
Lierl and Saloff-Coste \cite{MR3170207} studied
a general framework including Riemannian manifolds.
In that paper, they gave a precise heat kernel estimate for a bounded inner uniform domain,
which implies IU (\cite[Theorem 7.9]{MR3170207}).
In view of \cite{MR1124297}, however,
the requirement of inner uniformity for IU to hold can be relaxed.
See Section \ref{sec:rem}.

Our main result is  a sufficient condition for IU for domains in a manifold,
which is a generalization of the Euclidean case \cite{MR3420485}.
Our condition is given in terms of capacity.
It is applicable not only to bounded domains but also to unbounded domains.
Let $\Dom\subset M$ be an open set.
For $E\subset \Dom$ we define relative capacity by
\[
\capa_\Dom(E)
=\inf\Bigl\{\int_\Dom|\nabla\varphi|^2d\mu:
\text{ $\varphi\ge1$ on $E$, $\varphi\in C_0^\infty(\Dom)$}\Bigr\},
\]
where $\mu$ is the Riemannian measure in $M$
and $C_0^\infty(\Dom)$ is the space of all smooth functions compactly supported in $\dom$.
Let $d(x,y)$ be the distance between $x$ and $y$ in $M$.
The open geodesic ball with center $x$ and radius $r>0$ is
denoted by $B(x,r)=\{y\in M: d(x,y)<r\}$.
The closure of a set $E$ is denoted by $\cl E$,
and so $\cl B(x,r)$ stands for the closed geodesic ball of center $x$ and radius $r$.

\begin{defn}
Let $0<\eta<1$.
For an open set $\dom$ we define
the \emph{capacitary width} $\cw(\dom)$
by
\[
\cw(\dom)
=\inf\Big\{r>0:
  \frac{\capain{B(x,2r)}(\cl B(x,r)\sm \dom)}{\capain{B(x,2r)}(\cl B(x,r))}
  \ge\eta
  \qtext{for all $x\in \dom$} \Big\}.
\]
\end{defn}

The next theorem asserts that the parameter $\eta$ has no significance.

\begin{thm}\label{thm:cwcomp}Let $0<R_0<\infty$.
If $0<\eta'<\eta<1$, then
\[
w_{\eta'}(\dom)\le \cw(\dom)\le\const w_{\eta'} (\dom)
\qtext{for all open sets $\dom$ with $\cw(\dom)<R_0$}
\]
with $\const>1$ depending only on $\eta,\eta'$, $\sqrt K\, R_0$ and $\di$.
\end{thm}

The first condition for IU has a characterization in terms of capacitary width.
This is straightforward from Persson's argument \cite{MR0133586}, and \refthm{cwvDbot} below.
Hereafter we fix $o\in M$.

\begin{thm}\label{thm:nes}
Let $\dom$ be a domain in $M$.
Then $\dom$ has no essential spectrum if and only if
$\lim_{R\to\infty}\cw(\dom\sm \cl B(o,R))=0$.
\end{thm}

We shall show the following sufficient condition for IU,
which looks the same as in the Euclidean case \cite{MR3420485}.
Nevertheless, the proof is significantly different for negatively curved manifolds.
See the remark after \refthm{vdb}.

\begin{thm}\label{thm:IU}
Suppose $M$ has positive injectivity radius.
Then a domain $\dom\subset M$ is IU if the following two conditions are satisfied:
\begin{enumerate}
\item $\lim_{R\to\infty}\cw(\dom\sm \cl B(o,R))=0$.
\item For some $\tau>0$
\begin{equation}\label{eq:IUint}
\int_0^\tau \cw(\{x\in\dom: G_\dom(x,o)<t\})^2\frac{dt}t<\infty,
\end{equation}
where $G_\dom$ is the Green function for $\dom$.
\end{enumerate}
\end{thm}

Our results are based on the relationship between
the torsion function
\[
v_\dom(x)=\int_\dom G_\dom(x,y)d\mu(y)
\]
and the bottom of the spectrum
\begin{equation}\label{eq:botsD}
\bots(\dom)
=\inf\Biggl\{\frac{\normt{\nabla f}^2}{\normt{f}^2}:f \in \Czi(\dom)\text{ with } \normt{f}\ne0\Biggr\}.\end{equation}
We note that $\bots(\dom)$ is the first eigenvalue $\lambda_\dom$ if $\dom$ has no essential spectrum.
This is always the case for a bounded domain $\dom$.
\refthm{nes} asserts that the same holds even for an unbounded domain $\dom$ whenever
$\lim_{R\to\infty}\cw(\dom\sm \cl B(o,R))=0$.
We also observe that the torsion function is the solution to the de Saint-Venant problem:
\[
\begin{split}
-\Delta v_\dom=1 &\qtext{in }\dom,\\
v_\dom=0&\qtext{on }\bdy,
\end{split}
\]
where the boundary condition is taken in the Sobolev sense.
The second named author \cite{MR3682197} proved the following theorem.

\begin{knownthm}\label{thm:vdb}
Let $K=0$.
If $\dom\subset M$ satisfies $\bots(\dom)>0$,
then
\begin{equation}\label{eq:vdb}
\bots(\dom)\inv \le \normi{v_{\dom}} \le \const\bots(\dom)\inv,
\end{equation}
where $\const$ depends only on $M$.
\end{knownthm}

The second inequality of \refeq{vdb} does not necessarily hold for
negatively curved manifolds.
Let $\HH^\di$ be the $\di$-dimensional hyperbolic space of constant curvature $-1$.
It is known that
\[
\bots(\HH^\di)=\frac{(\di-1)^2}4,
\]
whereas  $v_{\HH^\di}\equiv\infty$ as $\HH^\di$ is stochastically complete.
Hence the second inequality of \refeq{vdb} fails to hold if $\dom$ is the whole space $\HH^\di$.

The point of this paper is that \refeq{vdb} still holds
if $\dom$ is limited to a certain class.
We make use of \refeq{vdb} with this limitation to derive
Theorems \ref{thm:nes} and \ref{thm:IU}.
We have the following theorem, which is a key ingredient in their proofs.

\begin{thm}\label{thm:cwvDbot}
Let $K\ge0$ and let $0<\eta<1$.
Then there exist $R_0>0$ and $\const>1$ depending only on $K$, $\eta$ and $\di$ such that if
$\dom\subset M$ satisfies $\cw(\dom)<R_0$, then
\begin{equation}\label{eq:cwvDbot}
\frac{\const\inv}{\cw(\dom)^2}
\le\frac1{\normi{v_\dom}}\le \bots(\dom)
\le\frac\const{\normi{v_\dom}}
\le\frac{\const^2}{\cw(\dom)^2}.
\end{equation}
\end{thm}

\begin{rem}\label{rem:L0}
We actually find $\Lambda_0>0$ depending only on $K$ and $\di$ such that
\refeq{vdb} holds for $\dom$ with $\bots(\dom)>\Lambda_0$ (\reflem{vDbot<C} below).
This is a generalization of \refthm{vdb} as $\Lambda_0=0$ for $K=0$.
In practice, however, the condition $\cw(\dom)<R_0$ in \refthm{cwvDbot} is more convenient since
the capacitary width $\cw(\dom)$ can be more easily estimated than the bottom of spectrum $\bots(\dom)$.\end{rem}

In Section \ref{sec:pre} we summarize
the key technical ingredients of the proofs:
the volume doubling property,
the Poincar\'e inequality and
the Li-Yau Gaussian estimate for the Dirichlet heat kernel for finite scale.
Observe that these fundamental tools are available not only for manifolds with
Ricci curvature bounded below by a negative constant
but also for unimodular Lie groups and homogeneous spaces.
See \cite[Example 2.11]{MR3170207} and \cite[Section 5.6]{MR1872526}.
This observation suggests that our approach is also extendable to those spaces.

We use the following notation.
By the symbol $\const$ we denote an absolute positive constant
whose value is unimportant and may change from one occurrence to the next.
If necessary, we use $\const_0, \const_1, \dots$, to specify them.
We say that $f$ and $g$ are comparable and write $f\approx g$
if two positive quantities $f$ and $g$ satisfy $\const\inv \le f/g\le\const$
with some constant $\const\ge1$.
The constant $\const$ is referred to as the constant of comparison.

\textbf{Acknowledgments.}
The authors would like to thank  the referee for his/her
careful reading of the manuscript and many useful suggestions.

\section{Preliminaries}\label{sec:pre}

We recall that $M$ is a manifold of dimension $\di\ge2$ with $\Ric\ge -K$ with $K\ge0$.
Let us recall  the volume doubling property of the Riemannian measure $\mu$,
the Poincar\'e inequality
and the Gaussian estimate for the Dirichlet heat kernel $p_M(t,x,y)$ for $M$.
For $B=B(x,r)$ and $\tau>0$ we write $\tau B=B(x,\tau r)$.

\begin{thm}[Volume doubling at finite scale. {\cite[Theorem 5.6.4]{MR1872526}}]\label{thm:VDR}
Let $0<R_0<\infty$.
Then for all $B=B(x,r)$ with $0<r<R_0$
\[
\mu(2B)\le 2^\di\exp\bigl(\sqrt{(\di-1)K}\, R_0\bigr)\mu(B).
\]
\end{thm}

\begin{thm}[Poincar\'e inequality {\cite[Theorem 5.6.6]{MR1872526}}]\label{thm:PIR}
For each $1\le p<\infty$ there exist positive constants $C_{\di,p}$ and $C_\di$ such that
\[
\int_B |f-f_B|^p d\mu \le C_{\di,p}r^p\exp(C_\di\sqrt K\, r)\int_{2B} |\nabla f|^pd\mu
\]
for all $B=B(x,r)$.
Here $f_B$ stands for the average of $f$ on $B$.
\end{thm}

\begin{cor}[Poincar\'e inequality at finite scale]\label{cor:PIR}
Let $0<R_0<\infty$.
Then for all $B=B(x,r)$ with $0<r<R_0$
\[
\int_B |f-f_B|^2 d\mu \le \const_{\di,2} r^2\exp(C_n\sqrt K\, R_0)\int_{2B} |\nabla f|^2d\mu.
\]
\end{cor}

\begin{rem}\label{rem:}
If the Ricci curvature of $M$ is nonnegative, i.e., $K=0$,
then the estimates in Theorems \ref{thm:VDR}, \ref{thm:PIR} and \refcor{PIR}
hold with constants independent of $0<r<\infty$.
\end{rem}

The Poincar\'e inequality yields the Sobolev inequality.
We see that if $B=B(x,r)$ with $0<r< R_0$, then
\[
\Biggl(\frac1{\mu(B)}\int_B|f|^{2}d\mu\Biggr)^{1/2}
\le \const_{\di,2} r\,\Biggl(\frac1{\mu(B)}\int_B|\nabla f|^{2}d\mu\Biggr)^{1/2}
\qtext{for all }f\in C_0^\infty(B)
\]
with different $\const_{\di,2}$.
See \cite[Theorem 5.3.3]{MR1872526} for more general Sobolev inequality.
Hence the characterization of the bottom of the spectrum
in terms of Rayleigh quotients \refeq{botsD} gives the following:

\begin{cor}\label{cor:lambdaBr}
Let $0<R_0<\infty$.
Then there exists a constant
$\const>0$ depending only on $\sqrt K\,R_0$ and $\di$ such that
\[
\bots(B(x,r))\ge \const r^{-2}
\qtext{for }0<r< R_0.
\]
\end{cor}

The celebrated theorem by Grigor'yan and Saloff-Coste
gives the relationship between
the Poincar\'e inequality,
the volume doubling property of the Riemannian measure,
the Li-Yau Gaussian estimate for the heat kernel,
and the parabolic Harnack inequality.
Let $V(x,r)=\mu(B(x,r))$.

\begin{knownthm}[{\cite[Theorems 5.5.1 and 5.5.3]{MR1872526}}]\label{thm:PI.VD.PHI.G}
Let $0<R_0\le\infty$.
Consider the following conditions:
\begin{enumerate}
\item (PI) There exists a constant $P_0>0$ such that
for all $B=B(x,r)$ with $0<r<R_0$ and all $f\in C^\infty(B)$,
\[
\int_B |f-f_B|^2 d\mu \le P_0r^2\int_{2 B} |\nabla f|^2d\mu.
\]

\item (VD) There exists  a constant $D_0>0$ such that
for all $B=B(x,r)$ with $0<r<R_0$
\[
\mu(2B)\le D_0\mu(B).
\]

\item (PHI) There exists  a constant $A>0$ such that for all $B=B(x,r)$ with $0<r<R_0$
and all $u>0$ with $(\partial_t-\Delta)u=0$ in $(s-r^2,s)\times B$
\[
\sup_{Q_-} u \le A\inf_{Q_+}u,
\]
where $Q_-=(s-\frac34r^2,s-\frac12r^2)\times B(x,\frac12r)$ and
$Q_+=(s-\frac14r^2,s)\times B(x,\frac12r)$.

\item (GE) There exists a finite constant $\const>1$ such that
for $0<t<R_0^2$ and $x,y\in M$,
\begin{equation}\label{eq:GER}
\frac{1}{\const V(x,\sqrt t)}\exp\Bigl(-\frac{\const d(x,y)^2}{t}\Bigr)
\le p_M(t,x,y)
\le \frac{\const}{V(x,\sqrt t)}\exp\Bigl(-\frac{d(x,y)^2}{\const t}\Bigr).
\end{equation}
\end{enumerate}
Then
\[
(i)+(ii)\iff (iii) \iff (iv).
\]
\end{knownthm}

\refthm{VDR} and \refcor{PIR} assert that
(i) and (ii) of \refthm{PI.VD.PHI.G} hold true for $0<R_0<\infty$
with constants depending only on $K$, $R_0$ and $\di$.
Hence,
the Li-Yau Gaussian estimate of the heat kernel for the whole manifold $M$
and the parabolic Harnack inequality up to scale $R_0$ are available
in our setting.
Observe that the volume doubling inequality $\mu(B(x,2r))\le D_0\mu(B(x,r))$ implies
\begin{equation}\label{eq:VDa}
\mu(B(x,r))\ge\const \Big(\frac rR\Big)^\alpha \mu(B(x,R))
\qtext{for }0<r<R<R_0
\end{equation}
with $\alpha=\log D_0/\log2$.
We also have the following elliptic Harnack inequality since
positive harmonic functions are time-independent positive solutions to the heat equation.

\begin{cor}[Elliptic Harnack inequality]\label{cor:eHI}
Let $0<r_1<r_2<R_0<\infty$.
If $h$ is a positive harmonic function in $B(x,r_2)$, then
\[
\const\inv\le \frac{h(y)}{h(x)} \le\const \qtext{for } y\in B(x,r_1)
\]
where $C>1$ depends only on $\sqrt{K}\,R_0$, $r_1/r_2$ and $\di$.
\end{cor}

\section{Torsion function and the bottom of spectrum}

In this section we obtain estimates between
the bottom of the spectrum and the torsion function $v_\dom$.
We shall show the second and the third inequalities of \refeq{cwvDbot}.

Since the Green function $G_\dom(x,y)$ is the integral of the heat kernel
$p_\dom(t,x,y)$ with respect to $t\in(0,\infty)$, we have
\[
v_\dom(x)
=\int_0^\infty  P_\dom(t,x) dt,
\]
where
\[
P_\dom(t,x)=\int_\dom p_\dom(t,x,y)d\mu(y).
\]
We note that $P_\dom(t,x)=\mathbb{P}_x[\tau_\dom>t]$,
i.e.,
the survival probability that
the Brownian motion $(B_t)_{t\ge0}$ started at $x$
stays in $\dom$ up to time $t$,
where $\tau_\dom$ is the first exit time from $\dom$.
We also observe that $P_\dom(t,x)$ is considered to be the (weak) solution to
\[
\begin{split}
\Big(\pder{}t -\Delta\Big)u(t,x)=0&\qtext{in }(0,\infty)\times\dom,\\
u(t,x)=0&\qtext{on }(0,\infty)\times\bdy,\\
u(0,x)=1&\qtext{on }\{0\}\times\dom.
\end{split}
\]
Let $\pi_\dom(t)=\sup_{x\in\dom} P_\dom(t,x)$.
Let us begin with the proof of the second inequality of \refeq{cwvDbot}.

\begin{lem}\label{lem:vDbot>1}
If $\bots(\dom)>0$, then
$\bots(\dom)\,\normi{v_\dom}\ge1$.
\end{lem}

\begin{proof}
We follow \cite[Lemmas 3.2 and 3.3]{MR3420485}.
Without loss of generality we may assume that $\normi{v_\dom}<\infty$.
It suffices to show the following two estimates:
\begin{align}
&
\exp(-\bots(\dom)\,t)\le\pi_\dom(t)
\qtext{for all $t>0$.}
\label{eq:vDbot>1a}\\
&\text{If $C>1$, then  }
\pi_\dom(t)\le\frac{C}{C-1}\exp\Big(-\frac{t}{C\normi{v_\dom}}\Big)
\qtext{for all $t>0$.}
\label{eq:vDbot>1b}
\end{align}
In fact, we obtain from \refeq{vDbot>1a} and \refeq{vDbot>1b} that
\[
\exp\Big(-\bots(\dom)\,t+\frac{t}{C\normi{v_\dom}}\Big)\le\frac{C}{C-1},
\]
which holds for all $t>0$ only if
\[
\bots(\dom)\ge \frac{1}{C\normi{v_\dom}}.
\]
Since $C>1$ is arbitrary, we have
$\bots(\dom)\,\normi{v_\dom}\ge1$.

Let us prove \refeq{vDbot>1a}.
Take $\alpha>\bots(\dom)$.
Then we find $\varphi\in\Czi(\dom)$
such that $\normt{\nabla\varphi}^2\big/\normt{\varphi}^2\le\alpha$.
Take a bounded domain $\Omega$ such that
$\supp\varphi\subset\Omega\subset\dom$.
Then $\Omega$ has no essential spectrum.
Let $\lambda_\Omega$ and $\varphi_\Omega$
be the first eigenvalue and its positive eigenfunction
with $\normt{\varphi_\Omega}=1$ for $\Omega$,
respectively.
By definition
\[
\lambda_\Omega
=\inf\Bigg\{\frac{\normt{\nabla\psi}^2}{\normt{\psi}^2}:\psi\in \Czi(\Omega)\Bigg\}
\le \frac{\normt{\nabla\varphi}^2}{\normt{\varphi}^2}
\le \alpha.
\]
Since
\(
u(t,x)=\exp(-\lambda_\Omega t)\,\varphi_\Omega(x)
\)
is the solution to the heat equation in $(0,\infty)\times\Omega$ such that
$u(0,x)=\varphi_\Omega(x)$ and $u(t,x)=0$
on $(0,\infty)\times\bd\Omega$,
it follows from the comparison principle that
\[
\exp(-\lambda_\Omega t)\,\varphi_\Omega(x)
\le\int_{\Omega} p_\dom(t,x,y)\varphi_\Omega(y)d\mu(y)
\le\normi{\varphi_\Omega}P_\dom(t,x)
\le\normi{\varphi_\Omega}\pi_\dom(t)
\]
in $(0,\infty)\times\Omega$.
Taking the supremum for $x\in\Omega$,
and then dividing by $0<\normi{\varphi_\Omega}<\infty$,
we obtain
\[
\exp(-\alpha t)
\le\exp(-\lambda_\Omega t)
\le\pi_\dom(t).
\]
Since $\alpha>\bots(\dom)$ is arbitrary,
we have \refeq{vDbot>1a}.

Let us show \refeq{vDbot>1b} to complete the proof of the lemma.
Let $C>1$ and $\beta=1/(C\normi{v_\dom})$.
Put
\[
w(t,x)=e^{-\beta t}(v_\dom(x)+(C-1)\normi{v_\dom}).
\]
Since $-\Delta v_\dom=1$ in $\dom$,
it follows that
\[
\begin{split}
\Big(\pder{}t-\Delta\Big)w
&=-\beta e^{-\beta t}(v_\dom+(C-1)\normi{v_\dom})-e^{-\beta t} \Delta v_\dom\\
&=e^{-\beta t}\Big(-\frac{v_\dom+(C-1)\normi{v_\dom}}{C\normi{v_\dom}}+1\Big)\\
&\ge e^{-\beta t}\Big(-\frac{\normi{v_\dom}+(C-1)\normi{v_\dom}}{C\normi{v_\dom}}+1\Big)
=0.
\end{split}
\]
Hence $w$ is a super solution to the heat equation.
By the comparison principle
\[
(C-1)\normi{v_\dom}P_\dom(t,x)
\le w(t,x)
=e^{-\beta t}(v_\dom(x)+(C-1)\normi{v_\dom})
\le Ce^{-\beta t}\normi{v_\dom}.
\]
Dividing the inequality by $0<\normi{v_\dom}<\infty$,
and taking the supremum for $x\in\dom$,
we obtain \refeq{vDbot>1b}.
\end{proof}

Next we prove the third inequality of \refeq{cwvDbot} under an additional assumption on $\bots(\dom)$.

\begin{lem}\label{lem:vDbot<C}
There exist $\Lambda_0>0$ and $\clabel{c:et}>0$ depending only on $K$ and $\di$ such that
if either $\bots(\dom)>\Lambda_0$ or $\normi{v_\dom}<1/\Lambda_0$, then
\begin{equation}\label{eq:et<C}
\bots(\dom)\,\normi{v_\dom}\le\refc{et}.
\end{equation}
\end{lem}

\begin{proof}
In view of \reflem{vDbot>1}, we see that $\normi{v_\dom}<1/\Lambda_0$ implies $\bots(\dom)>\Lambda_0$.
So, it suffices to show \refeq{et<C} under the assumption $\bots(\dom)>\Lambda_0$ with
$\Lambda_0$ to be determined later.

For simplicity we write $\lambda_\dom$ for $\bots(\dom)$, albeit $\bots(\dom)$
need not be an eigenvalue.
Let $0<R_0<\infty$.
By symmetry, the Gaussian estimate \refeq{GER} implies
\begin{equation}\label{eq:GER2}
\begin{split}
\frac{1}{\const V(x,\sqrt t)^{1/2}V(y,\sqrt t)^{1/2}}&\exp\Bigl(-\frac{\const d(x,y)^2}{t}\Bigr)
\le p_M(t,x,y)\\
&\le \frac{\const}{V(x,\sqrt t)^{1/2}V(y,\sqrt t)^{1/2}}\exp\Bigl(-\frac{d(x,y)^2}{\const t}\Bigr)
\end{split}
\end{equation}
with the same $\const$;
and conversely,  \refeq{GER2} implies \refeq{GER}
with different $\const$ depending only on $\sqrt K\, R_0$ and $\di$
by volume doubling.
Let $0<t<R_0^2$.
By \cite[Exercise 10.29]{MR2569498} we have
\[
\begin{split}
p_\dom(t,x,y)
&\le p_\dom(t,x,y)^{1/2}p_M(t,x,y)^{1/2}\\
&\le\Bigl(e^{-\lambda_\dom t/2}\sqrt{p_\dom(t/2,x,x)p_\dom(t/2,y,y)}\Bigr)^{1/2}
    p_M(t,x,y)^{1/2}\\
&\le e^{-\lambda_\dom t/4}p_M(t/2,x,x)^{1/4}p_M(t/2,y,y)^{1/4}p_M(t,x,y)^{1/2},
\end{split}
\]
so that the upper estimates of \refeq{GER} and \refeq{GER2}, together with volume doubling,
show that $p_\dom(t,x,y)$ is bounded by
\[
\begin{split}
e^{-\lambda_\dom t/4}\Big\{\frac\const{V(x,\sqrt{t/2})}\Big\}^{1/4}
&  \cdot\Big\{\frac\const{V(y,\sqrt{t/2})}\Big\}^{1/4}\cdot
  \Bigl\{\frac\const{V(x,\sqrt{t})^{1/2}V(y,\sqrt{t})^{1/2}}
\exp\Bigl(-\frac{d(x,y)^2}{\const t}\Bigr)\Bigr\}^{1/2}\\
&\le e^{-\lambda_\dom t/4} \frac{\const\const'}{V(x,\sqrt{t})^{1/2}V(y,\sqrt{t})^{1/2}}
     \exp\Bigl(-\frac{d(x,y)^2}{2\const t}\Bigr),
\end{split}
\]
where  $\const'$ takes care of the various volume doubling factors.
By the lower estimate of \refeq{GER2} with $2\const^2t$ in place of $t$ and volume doubling,
we find $\clabel{c:PDom}\ge1$ depending only on $\sqrt K\,R_0$ and $\di$
such that
\[
p_\dom(t,x,y)\le \refc{PDom}e^{-\lambda_\dom t/4} p_M(2\const^2t,x,y).
\]
Integrating the inequality with respect to $y\in\dom$, we obtain
\[
P_\dom(t,x)=
\int_\dom p_\dom(t,x,y)d\mu(y)
\le\refc{PDom} e^{-\lambda_\dom t/4} \int_\dom p_M(2\const^2t,x,y)d\mu(y)
\le\refc{PDom} e^{-\lambda_\dom t/4}.
\]
Taking the supremum over $x\in\dom$, we obtain
\begin{equation}\label{eq:piD}
\pi_\dom(t)\le \refc{PDom}\exp\Bigl(-\frac{\lambda_\dom t}4\Bigr)
\qtext{for }0<t<R_0^2.
\end{equation}

Let $T=R_0^2/2$.
We claim that \refeq{et<C} holds with $\refc{et}=8\log(2\refc{PDom})$,
and with $\Lambda_0=4T\inv\log(2\refc{PDom})$ or
\begin{equation}\label{eq:lam0}
\refc{PDom}\exp\Bigl(-\frac{\Lambda_0 T}4\Bigr)=\frac12.
\end{equation}
Suppose $\lambda_\dom>\Lambda_0$.
Then \refeq{piD} with $t=T$ yields $\pi_\dom(T)\le1/2$.
Solving the initial value problem from time $T$,
we see that
\[
P_\dom(t,x)\le \pi_\dom(T)\cdot P_\dom(t-T,x)\le \frac12 \qtext{for }t\ge T.
\]
Take the supremum for $x\in\dom$.
We find
\[
\pi_\dom(t)\le \frac12 \qtext{for }t\ge T.
\]
Repeating the same argument, we obtain
\[
\pi_\dom(t)\le \frac1{2^k} \qtext{for $kT \le t< (k+1)T$ with $k=0,1,2,\dots$.}
\]
Hence
\[
\begin{split}
v_\dom(x)
&=\int_0^\infty P_\dom(t,x)dt
=\sum_{k=0}^\infty \int_{kT}^{(k+1)T} P_\dom(t,x)dt\\
&\le\sum_{k=0}^\infty \int_{kT}^{(k+1)T} \pi_\dom(t)dt
\le T \sum_{k=0}^\infty\frac1{2^k}=2T
\le\dfrac{2\Lambda_0 T}{\lambda_\dom}
=\dfrac{8\log(2\refc{PDom})}{\lambda_\dom}
\end{split}
\]
by \refeq{lam0}.
Taking the supremum for $x\in\dom$, we obtain
\(
\lambda_\dom\normi{v_\dom}
\le 8\log(2\refc{PDom}),
\)
as required.
\end{proof}

\begin{rem}\label{rem:eigen.torsion}
If the Gaussian estimate \refeq{GER}
holds uniformly for all $0<t<\infty$, then
there exists $\const>0$ such that
$\bots(\dom)\,\normi{v_\dom}\le\const$ for all $\dom\subset M$.
This is the case when $K=0$.
See \cite{MR3682197}.
\end{rem}

\section{Capacitary width and harmonic measure}

By $\hm[x]{E}{\dom}$ we denote
the harmonic measure of $E$ in $\dom$ evaluated at $x$.
In this section we give an estimate for harmonic measure in terms of capacitary width.
This will be crucial for the proof of \refthm{cwcomp}.

\begin{thm}[cf. {\cite[Theorem 12.7]{MR3420485}}]\label{thm:hmebycw}
Let $0<R_0<\infty$.
Let $\dom\subset M$ be an open set with $\cw(\dom)<R_0$.
If $x\in \dom$ and $R>0$, then
\[
\hm[x]{\dom\cap\bd B(x,R)}{\dom\cap B(x,R)}
\le\exp\Big(2\refc{cwe}-\frac{\refc{cwe}R}{\cw(\dom)}\Big),
\]
where $\clabel{c:cwe}$
depends only on $\sqrt K\,R_0$, $\eta$ and $\di$.
\end{thm}

Let us begin by estimating the torsion function of a ball.

\begin{lem}\label{lem:vB}
Let $0<R_0<\infty$.
Then there exists a constant $\const>1$ depending only on
$\sqrt K\,R_0$ and $\di$ such that
\[
\const\inv r^2\le \normi{v_{B(x,r)}}\le \const r^2 \qtext{for } 0<r< R_0.
\]
\end{lem}

\begin{proof}
Let $0<r< R_0$.
Write $B=B(x,r)$ for simplicity.
We have $\bots(B)\ge \const r^{-2}$ by \refcor{lambdaBr}.
Since $B$ is bounded, the bottom of spectrum  is an eigenvalue.
So let us write $\lambda_B$ for $\bots(B)$.
Let $z\in B$.
In view of \cite[Exercise 10.29]{MR2569498}, \refeq{GER} and the volume doubling property, we have
\[
\begin{split}
v_B(z)&=\int_B G_B(z,y)d\mu(y)=\int_0^\infty dt\int_B p_B(t,z,y)d\mu(y)\\
&=\int_0^{r^2} dt\int_B p_B(t,z,y)d\mu(y)+\int_{r^2}^\infty dt\int_B p_B(t,z,y)d\mu(y)\\
&\le r^2+\int_{r^2}^\infty e^{-\lambda_B(t-r^2)}dt
   \int_B \sqrt{p_B(r^2,z,z)p_B(r^2,y,y)}\,d\mu(y)\\
&\le r^2+\frac1{\lambda_B} \int_B \frac{\const d\mu(y)}{\sqrt{V(z,r)V(y,r)}}
\le r^2+\const r^2,
\end{split}
\]
where $\const$ depends only on $\sqrt K\,R_0$ and $\di$.
Hence $\normi{v_B}\le\const r^2$.

The  opposite inequality is an immediate consequence of the combination of
\refcor{lambdaBr} and \reflem{vDbot>1}.
But for later purpose we give a direct proof based on a lower estimate
of the Dirichlet heat kernel of a ball: if $x\in M$, then
\[
p_B(t,y,z)\ge\frac\const{V(x,\sqrt t)}
\qtext{for $y,z\in\ve B$ and $0<t<\ve r^2$}
\]
valid for some $0<\ve<1$ and $\const>0$.
In fact, this lower estimate is equivalent to the Gaussian estimate \refeq{GER}.
See  e.g. \cite[(1.5)]{MR2998918}.
If $y\in\ve B$, then
\[
v_B(y)
=\int_B G_B(y,z)d\mu(z)
\ge \int_{0}^{\ve r^2} dt\int_{\ve B} p_B(t,y,z)d\mu(z)
\ge \frac{\ve r^2\const\mu(\ve B)}{V(x,\sqrt{\ve}r)}\ge \const r^2
\]
by volume doubling.
Thus $\normi{v_B}\ge\const r^2$.
\end{proof}

For later  use we record the above estimate: if $0<r<R_0$, then
\begin{equation}\label{eq:vB>r2}
v_{B(x,r)}\ge\refc{vB>r2} r^2 \qtext{on } B(x,\ve r),
\end{equation}
where $\ve$ and $\clabel{c:vB>r2}$ depends only on $\sqrt K\,R_0$ and $\di$.

\begin{rem}\label{rem:vB}
In case $K>0$, the inequality \refeq{vB>r2} does not necessarily hold for all $0<r<\infty$ uniformly.
Let $\HH^\di$ be the $\di$-dimensional hyperbolic space of constant curvature $-1$.
Then the torsion function for $B(a,r)$ is a radial function $f(\rho)$ of $\rho=d(x,a)$
satisfying
\[
-1=\Delta f(\rho)
=\frac1{(\sinh\rho)^{\di-1}}\frac{d}{d\rho}\Big\{(\sinh\rho)^{\di-1}\frac{df}{d\rho}\Big\}
\qtext{for $0<\rho<r$,}
\]
$f(r)=0$, $f'(0)=0$ and $f(0)=\normi{v_{B(a,r)}}$.
See \cite[pp.176-177]{MR990239} or \cite[(3.85)]{MR2569498}.
Hence
\[
\normi{v_{B(a,r)}}
=\int_0^r \int_0^\rho \Big(\frac{\sinh t}{\sinh \rho}\Big)^{\di-1}dt d\rho.
\]
Since the integrand is less than 1, we have
\(
\normi{v_{B(a,r)}}\le\frac12r^2
\)
for all $r>0$.
Observe that $t\le\sinh t$ for $t>0$ and $\sinh\rho \le \rho\cosh R_0$ for $0<\rho<R_0$.
Hence, if $0<r<R_0$, then
\[
\normi{v_{B(a,r)}}
\ge\int_0^r \int_0^\rho \Big(\frac{t}{\rho\cosh R_0}\Big)^{\di-1}dt d\rho
=\frac{r^2}{2\di(\cosh R_0)^{\di-1}},
\]
so that $\normi{v_{B(a,r)}}\approx r^2$.
This gives the estimate in \reflem{vB} with explicit bounds.

On the other hand, if $r>1$, then $\sinh \rho\ge\frac12(1-e^{-2})e^\rho$ for $1<\rho<r$, so that
\[
\begin{split}
\normi{v_{B(a,r)}}
&\le\int_0^1 \int_0^\rho dtd\rho
+\int_1^r \int_0^\rho \Big(\frac{\sinh t}{\sinh \rho}\Big)^{\di-1}dt d\rho
\le\frac12+\int_1^r \int_0^\rho \Big(\frac{e^t}{(1-e^{-2})e^\rho}\Big)^{\di-1}dt d\rho\\
&=\frac12+\frac1{\di-1}\int_1^r\frac{e^{(\di-1)\rho}-1}{((1-e^{-2})e^\rho)^{\di-1}}d\rho
\le \frac12+\frac{r-1}{(\di-1)(1-e^{-2})^{\di-1}}.
\end{split}
\]
Thus $\normi{v_{B(a,r)}}=O(r)$ as $r\to\infty$,
so \refeq{vB>r2} fails to hold uniformly for $0<r<\infty$.
This example illustrates that the assumption $0<r<R_0$
cannot be dropped in \reflem{vB}.
\end{rem}

Next we compare capacity and volume.
Observe that $\capa_\dom(E)$ coincides with the Green capacity
of $E$ with respect to $\dom$, i.e.,
\begin{equation}\label{eq:capsup}
\capa_\dom(E)
=\sup\Bigl\{\mass{\nu}: \supp\nu\subset E
\text{ and }\int_\dom G_\dom(x,y)d\nu(y)\le1\text{ on } \dom\Bigr\},
\end{equation}
where $\mass{\nu}$ stands for the total mass of the measure $\nu$.

\begin{lem}\label{lem:v<c}
Let $0<R_0<\infty$.
There exists a constant $\clabel{c:v<c}>0$ depending only on $\sqrt K\,R_0$ and $\di$
such that if $0<r< R_0$, then
\[
\frac{\mu(E)}{\mu(\clB(x,r))}
\le \refc{v<c}\frac{\capa_{B(x,2 r)}(E)}{\capa_{B(x,2 r)}(\clB(x,r))}
\]
for every Borel set $E\subset \clB(x,r)$.
\end{lem}

\begin{proof}
Let $0<r<R_0$.
\reflem{vB} yields
\[
\int_{E}G_{B(x,2 r)}(y,z)d\mu(z)
\le \int_{B(x,2 r)}G_{B(x,2 r)}(y,z)d\mu(z)
\le \normi{v_{B(x,2r)}}
\le \const r^2
\qtext{for all }y\in M,
\]
where $\const$ depends only on $\sqrt K\,R_0$ and $\di$.
Hence the characterization \refeq{capsup} of capacity gives
\begin{equation}\label{eq:CE>r-2mE}
\capa_{B(x,2 r)}(E)\ge\frac{\mu(E)}{\const r^2}.
\end{equation}
Let $\varphi(y)=\min\{2-{d(y,x)}/r,1\}$.
Observe that $\varphi\in W_0^1(B(x,2 r))$,
$|\nabla\varphi|\le 1/r$ and $\varphi=1$ on $\clB(x,r)$.
The definition of capacity and the volume doubling property yield
\[
\capa_{B(x,2 r)}(\clB(x,r))
\le \int_{B(x,2 r)}|\nabla\varphi|^2d\mu
\le \frac{\mu(\clB(x,2r))}{r^2}
\le \frac{\const\mu(\clB(x,r))}{r^2}.
\]
This, together with \refeq{CE>r-2mE} for $E=\clB(x,r)$, shows that
$\capa_{B(x,2 r)}(\clB(x,r))\approx r^{-2}\mu(\clB(x,r))$
with the constant of comparison depending only on $\sqrt K\,R_0$ and $\di$.
Dividing \refeq{CE>r-2mE} by $\capa_{B(x,2 r)}(\clB(x,r))$,
we obtain the lemma.
\end{proof}

Let us introduce regularized reduced functions,
which are closely related to capacity and harmonic measure.
See \cite[Sections 5.3-7]{MR1801253} for the Euclidean case.
Let $\dom$ be an open set.
For $E\subset \dom$ and a nonnegative function $u$ in $E$,
we define the reduced function $\Red[\dom]{E}{u}$ by
\[
\Red[\dom]{E}{u}(x)
=\inf\{v(x):
\text{ $v\ge0$ is superharmonic in $\dom$ and $v\ge u$ on $E$}\}
\qtext{for }x\in \dom.
\]
The lower semicontinuous regularization of $\Red[\dom]{E}{u}$ is called
the \emph{regularized reduced function} or \emph{balayage} and is denoted by
$\red[\dom]{E}{u}$.
It is known that $\red[\dom]{E}{u}$ is a nonnegative superharmonic function,
$\red[\dom]{E}{u}\le\Red[\dom]{E}{u}$ in $\dom$
with equality outside a polar set.
If $u$ is a nonnegative superharmonic function in $\dom$,
then $\red[\dom]{E}{u}\le u$ in $\dom$.
By the maximum principle $\red[\dom]{E}{u}$ is
nondecreasing with respect to $\dom$ and $E$.
If $u$ is the constant function $1$,
then $\red[\dom]{E}{1}(x)$ is the probability of
Brownian motion hitting $E$ before leaving $\dom$ when it starts at $x$.
In an almost verbatim way we can extend \cite[Lemma F]{MR3420485} to the present setting.
But, for  completeness, we shall provide a proof.

\begin{lem}\label{lem:capa=infred}
Let $0<r<R<R_0<\infty$.
\begin{enumerate}
\item
\(\ds
\inf_{\clB(x,r)} \red[B(x,R)]E1
\le\frac{\capain{B(x,R)}(E)}{\capain{B(x,R)}(\clB(x,r))}
\)
\quad  for $E\subset B(x,R)$.
\item
\(\ds
\frac{\capain{B(x,R)}(E)}{\capain{B(x,R)}(\clB(x,r))}
\le\const\inf_{\clB(x,r)} \red[B(x,R)]E1
\)
\quad
for $E\subset \clB(x,r)$
with $\const>1$ depending only on $\sqrt K\,R_0$, $r/R$ and $\di$.
\end{enumerate}
\end{lem}

\begin{proof}
Let $\nu_E$ and $\nu_B$ be the capacitary measures of $E$
and $\clB(x,r)$, respectively.
Then $\nu_E$ is supported on $\cl E$,
$G_{B(x,R)}\nu_E=\red[B(x,R)]E1$ and $\mass{\nu_E}=\capain{B(x,R)}(E)$;
$\nu_B$ is supported on $\cl B(x,r)$,
$G_{B(x,R)}\nu_B=\red[B(x,R)]{\clB(x,r)}1$ and $\mass{\nu_B}=\capain{B(x,R)}(\clB(x,r))$.
In particular, $G_{B(x,R)}\nu_B \le 1$ in $B(x,R)$ and hence
\[
\begin{split}
\capain{B(x,R)}(E)
& \ge \int G_{B(x,R)}\nu_B d\nu_E
  = \int G_{B(x,R)}\nu_Ed\nu_B = \int \red[B(x,R)]E1 d\nu_B \\
& \ge \int \Bigl(\inf_{\clB(x,r)} \red[B(x,R)]E1\Bigr)d\nu_B
= \Bigl(\inf_{\clB(x,r)} \red[B(x,R)]E1 \Bigr) \capain{B(x,R)}(\clB(x,r)).
\end{split}
\]
Thus (i) follows.

Let $\rho=(r+R)/2$.
The elliptic Harnack inequality (\refcor{eHI}) implies
\begin{align*}
&G_{B(x,R)}(z,y)\approx G_{B(x,R)}(z,x)\qtext{for $z\in\bd B(x,\rho)$ and $y\in \clB(x,r)$,}\\
&\red[B(x,R)]{\clB(x,r)}1\approx 1 \qtext{on $\bd B(x,\rho)$,}
\end{align*}
where, and hereafter, the constants of comparison depend only on $\sqrt K\,R_0$, $r/R$ and $\di$.
Let $E\subset\clB(x,r)$.
Since $\supp\nu_E\subset\clB(x,r)$, we have for $z\in\bd B(x,\rho)$,
\begin{align*}
&\red[B(x,R)]E1(z)
=\int G_{B(x,R)}(z,y)d\nu_E(y)
\approx G_{B(x,R)}(z,x)\capain{B(x,R)}(E),\\
&\red[B(x,R)]{\clB(x,r)}1(z)=\int G_{B(x,R)}(z,y)d\nu_B(y)
\approx G_{B(x,R)}(z,x)\capain{B(x,R)}(\clB(x,r)),
\end{align*}
so that
\[
\frac{\capain{B(x,R)}(E)}{\capain{B(x,R)}(\clB(x,r))}
\approx \red[B(x,R)]E1(z).
\]
Since $z\in\bd B(x,\rho)$ is arbitrary,
the superharmonicity of $\red[B(x,R)]E1$ and the maximum principle yield (ii).
\end{proof}

We restate the above lemma in terms of harmonic measure.
We recall  $\hm[x]{E}{\dom}$ stands for
the harmonic measure of $E$ in $\dom$ evaluated at $x$.
We see that if $E$ is a compact subset of $B(x,R)$, then
\begin{equation}\label{eq:1-redBE1}
\hm{\bd B(x,R)}{B(x,R)\sm E}=1-\red[B(x,R)]E1
\qtext{on }B(x,R).
\end{equation}
Strictly speaking, the harmonic measure is extended by the right-hand side.
\reflem{capa=infred} reads as follows.

\begin{lem}\label{lem:1-capa=hm}
Let $0<r<R<R_0<\infty$.
\begin{enumerate}
\item
\(\ds
1- \frac{\capain{B(x,R)}(E)}{\capain{B(x,R)}(\clB(x,r))}
\le \sup_{\clB(x,r)} \hm{\bd B(x,R)}{B(x,R)\sm E}
\)
\quad for $E\subset B(x,R)$.

\item
\(\ds
\sup_{\clB(x,r)} \hm{\bd B(x,R)}{B(x,R)\sm E}
\le 1- \const\inv\frac{\capain{B(x,R)}(E)}{\capain{B(x,R)}(\clB(x,r))}
\)
\quad for $E\subset\clB(x,r)$
with $\const>1$ depending only on $\sqrt K\,R_0$, $r/R$ and $\di$.
In particular, if $0<r<R_0/2$, then
\[
\sup_{\clB(x,r)} \hm{\bd B(x,2r)}{B(x,2r)\sm E}
\le 1- \refc{1-capa=hm}\inv\frac{\capain{B(x,2r)}(E)}{\capain{B(x,2r)}(\clB(x,r))},
\]
where $\clabel{c:1-capa=hm}>1$ depends only on $\sqrt K\,R_0$ and $\di$.
\end{enumerate}
\end{lem}

Applying \reflem{1-capa=hm} repeatedly, we obtain the following estimate of harmonic measure,
which is a preliminary version of \refthm{hmebycw}.

\begin{lem}\label{lem:Omega<(1-Aeta)k}
Let $0<R_0<\infty$.
Let $\dom\subset M$ be an open set with $\cw(\dom)<R_0$.
Suppose $x\in \dom$ and $R>0$.
If $k$ is a nonnegative integer such that $R-2k\cw(\dom)>0$, then
\[
\sup_{\dom\cap\clB(x,R-2k\cw(\dom))}\hm{\dom\cap\bd B(x,R)}{\dom\cap B(x,R)}
\le(1-\refc{1-capa=hm}\inv\eta)^k.
\]
\end{lem}

\begin{proof}
For simplicity let $\omega_0=\hm{\dom\cap\bd B(x,R)}{\dom\cap B(x,R)}$.
By definition we find $r>\cw(\dom)$ arbitrarily close to $\cw(\dom)$ such that
\[
\frac{\capain{B(y,2r)}(\clB(y,r)\sm \dom)}{\capain{B(y,2r)}(\clB(y,r))}\ge\eta
\qtext{for all }y\in \dom.
\]
Hence it suffices to show that $\omega_0\le(1-\refc{1-capa=hm}\inv\eta)^k$
in $\dom\cap\clB(x,R-2kr)$.
Let us prove this inequality  by induction on $k$.
The case $k=0$ holds trivially.
Let $k\ge1$ and suppose
$\omega_0\le(1-\refc{1-capa=hm}\inv\eta)^{k-1}$
in $\dom\cap\clB(x,R-2(k-1)r)$.
Take $y\in \dom\cap\bd B(x,R-2kr)$ and let $E=\clB(y,r)\sm \dom$.
Since $\dom\cap B(y,2r)\subset \dom\cap\clB(x,R-2(k-1)r)$, we have
\[
\begin{split}
\omega_0
&\le(1-\refc{1-capa=hm}\inv\eta)^{k-1}\hm{\dom\cap\bd B(y,2r)}{\dom\cap B(y,2r)}\\
&\le(1-\refc{1-capa=hm}\inv\eta)^{k-1}\hm{\bd B(y,2r)}{\dom\sm E}
\le (1-\refc{1-capa=hm}\inv\eta)^{k}
\end{split}
\]
in $\dom\cap B(y,2r)$.
Since $y\in \dom\cap\bd B(x,R-2kr)$ is arbitrary, we have
$\omega_0\le(1-\refc{1-capa=hm}\inv\eta)^{k}$ on $\dom\cap\bd B(x,R-2kr)$,
and hence in $\dom\cap \clB(x,R-2kr)$ by the maximum principle,
as required.
\end{proof}

This lemma and the definition of capacitary width yield

\begin{proof}[Proof of \refthm{hmebycw}]
Let $k$ be the integer such that
\(
2k\cw(\dom)<R\le2(k+1)\cw(\dom).
\)
\reflem{Omega<(1-Aeta)k} gives
\[
\begin{split}
\hm[x]{\dom\cap\bd B(x,R)}{\dom\cap B(x,R)}
&\le(1-\refc{1-capa=hm}\inv\eta)^k
=\exp\Big(-k\log\frac1{1-\refc{1-capa=hm}\inv\eta}\Big)\\
&\le\exp\Bigg(-\Big(\frac{R}{2\cw(\dom)}-1\Big)
\log\frac1{1-\refc{1-capa=hm}\inv\eta}\Bigg),
\end{split}
\]
which implies the required inequality
with
\[
\refc{cwe}=\frac12\log\frac1{1-\refc{1-capa=hm}\inv\eta}.
\qedhere
\]
\end{proof}

\section{Proofs of Theorems \ref{thm:cwcomp} and \ref{thm:cwvDbot}}

In this section we prove \refthm{cwcomp} and complete the proof of \refthm{cwvDbot}
by showing

\begin{thm}\label{thm:vd=cw2}
Let $0<R_0<\infty$.
If $\cw(\dom)<R_0$, then
\begin{equation}\label{eq:cw<G<cw}
\const\inv\cw(\dom)^2\le \normi{v_\dom}\le \const\cw(\dom)^2
\end{equation}
where $\const$ depends only on $\sqrt K\, R_0$, $\eta$ and $\di$.
\end{thm}

This theorem, together with \refeq{vDbot>1b} in \reflem{vDbot>1},
immediately yields the following estimate of the survival probability,
which plays a crucial role in the proof of \refthm{IU}.

\begin{thm}\label{thm:P<exp(A/w2)}
Let $0<R_0<\infty$.
There exist positive constants
$\clabel{c:ke1}$ and $\clabel{c:ke2}$
depending only on $\sqrt K\, R_0$, $\eta$ and $\di$
such that
\begin{equation}\label{eq:ke}
P_\dom(t,x)
\le\refc{ke1}\exp\Big(-\frac{\refc{ke2}t}{\cw(\dom)^2}\Big)
\qtext{for all $t>0$ and $x\in\dom$,}
\end{equation}
whenever $\cw(\dom)<R_0$.
\end{thm}

Let us begin with a uniform estimate of the capacity of balls.

\begin{lem}\label{lem:capaB}
Let $0<R_0<\infty$.
For $0<t\le1$,
define
\[
\kappa(t)
=\inf\left\{\frac{\capain{B(x,2R)}(\clB(x,tR))}{\capain{B(x,2R)}(\clB(x,R))}: x\in M,\ 0<R<R_0\right\}.
\]
Then $\lim_{t\to1}\kappa(t)=1$.
\end{lem}

\begin{proof}
Without loss of generality we may assume that $\frac12<t\le1$.
Let $\Dom=B(x,2R)\sm \cl B(x,tR)$ and let $E_t=\bd B(x, tR)$.
We find $a>0$ such that
for each $y\in E_t$ and $0<r<\frac14R$ there exists a ball of radius $ar$ lying in $B(y,r)\sm \Dom$.
This means that
\[
\frac{\mu(B(y,r)\sm\Dom)}{\mu(B(y,r))}\ge\ve
\]
with some $\ve>0$ depending only on $a$ and the doubling constant.
By Lemmas \ref{lem:v<c} and \ref{lem:1-capa=hm} we have
\begin{equation}\label{eq:<1-e}
\sup_{\clB(y,r)} \hm{\bd B(y,2r)}{B(y,2r)\cap \Dom}\le 1-\ve'
\end{equation}
with $\ve'>0$ independent of $x,\ R,\ t,\ y$ and $r$.

The technique in the proof of \cite[Theorem 1]{MR856511} yields
a positive superharmonic function $s$ in $\Dom$ such that
\begin{equation}\label{eq:s=da}
s\approx\dist(\cdot,E_t)^\alpha,
\end{equation}
where
$\alpha>0$ and the constants of comparison are independent of $x,\ R$ and $t$.
In fact, let $r_k=4^k$, $k\in\Z$. For each $k\in\Z$ choose a locally finite covering of
$E_t$ by open balls $B(x_{kj},\frac14r_k)$, $j\in J_k$; let $B_{kj}=B(x_{kj},r_k)$.
By \refeq{<1-e} we find a positive continuous function $u_{kj}$ in $\Dom\cap\cl B_{kj}$,
superharmonic in $\Dom\cap B_{kj}$, such that
$\ve''\le u_{kj}\le2$ in $\Dom\cap B_{kj}$,
$u_{kj}\ge1$ in $\Dom\cap\bd B_{kj}$,
$u_{kj}\le1-\ve''$ in $\Dom\cap\frac12B_{kj}$,
where $\ve''$ is a small positive constant depending only on $\ve'$.
Let $A=1-\frac12\ve''$ and extend $u_{kj}$ on $\Dom\sm\cl B_{kj}$ by $u_{kj}=\infty$.
Then
\[
s(x)=\inf\{A^{-k}u_{kj}(x): k\in\Z,\ j\in J_k\},\quad x\in\Dom
\]
is a superharmonic function in $\Dom$ satisfying \refeq{s=da} with $\alpha=|\log A|/\log4$.
Actually, we can make $s$ a strong barrier.
In the present context, however, superharmonicity is enough.

From \refeq{s=da}, we find a positive constant $\const$ independent of $x,\ R$ and $t$ such that
\[
\frac{s}{\const R^\alpha}\ge1\qtext{on } \bd B(x,3R/2).
\]
Let $u$ be the capacitary potential for $\clB(x,tR)$ in $B(x,2R)$, i.e.,
\[
\begin{split}
&\Delta u=0 \qtext{in }B(x,2R)\sm \clB(x,tR),\\
&u=1 \qtext{on } \clB(x,tR),\\
&u=0 \qtext{on } \bd B(x,2R),\\
&\capain{B(x,2R)}(\clB(x,tR))=\int_{B(x,2R)}|\nabla u|^2d\mu.
\end{split}
\]
Since $1-u\le s/(\const R^\alpha)$ on $\bd B(x,3R/2)$,
it follows from the maximum principle
\[
1-u\le \frac s{\const R^\alpha}
\approx \frac{\dist(\cdot,E_t)^\alpha}{R^\alpha}\qtext{in } B(x,3R/2)\sm \cl B(x,tR).
\]
Hence
\[
u\ge 1-\const\frac{((1-t)R)^\alpha}{R^\alpha}=1-\const(1-t)^\alpha \qtext{in } B(x,R)\sm \cl B(x,tR)
\]
with another positive constant $\const$.
If $1-\const(1-t)^\alpha>0$, then by definition,
\[
\capain{B(x,2R)}(\clB(x,R))
\le\frac1{(1-\const(1-t)^\alpha)^2} \int_{B(x,2R)}|\nabla u|^2d\mu
= \frac{\capain{B(x,2R)}(\clB(x,tR))}{(1-\const(1-t)^\alpha)^2}.
\]
Hence
\[
\frac{\capain{B(x,2R)}(\clB(x,tR))}{\capain{B(x,2R)}(\clB(x,R))}
\ge {(1-\const(1-t)^\alpha)^2},
\]
so that the lemma follows as $\lim_{t\to1}(1-\const(1-t)^\alpha)^2=1$.
\end{proof}

\begin{proof}[Proof of \refthm{cwcomp}]
By definition the first inequality holds for arbitrary open sets $\dom$.
Let us prove the second inequality.
In view of \reflem{capaB},
we find an integer $N\ge2$ depending only on $\sqrt K\,R_0$ and $\di$ such that
\begin{equation}\label{eq:cap1-Ninv}
\frac{\capain{B(x,2R)}(\clB(x,(1-N\inv)R))}{\capain{B(x,2R)}(\clB(x,R))}
\ge \sqrt{\eta}
\end{equation}
uniformly for $x\in M$ and $0<R<R_0$.
Let $\refc{1-capa=hm}$ be as in \reflem{1-capa=hm} and
take an integer $k>2$ so large that
$(1-\refc{1-capa=hm}\inv\eta')^k\le 1-\sqrt{\eta}$.

Let $\cw(\dom)<R_0$. We prove the theorem by showing
\begin{equation}\label{eq:cw<2Nkcw'}
\cw(\dom)\le 2Nk w_{\eta'}(\dom).
\end{equation}
If
\(
w_{\eta'}(\dom)\ge R_0/(2Nk),
\)
then $\cw(\dom)<R_0\le 2Nk w_{\eta'}(\dom)$, so \refeq{cw<2Nkcw'} follows.
Suppose
\[
w_{\eta'}(\dom)< \frac{R_0}{2Nk}.
\]
For simplicity we write $\rho=w_{\eta'}(\dom)$.
Apply \reflem{Omega<(1-Aeta)k}, with $\eta'$ in place of $\eta$, to $x\in \dom$ and $R=2Nk\rho$.
We obtain
\[
\sup_{\dom\cap\clB(x,R-2k\rho)}\hm{\dom\cap\bd B(x,R)}{\dom\cap B(x,R)}
\le(1-\refc{1-capa=hm}\inv\eta')^k
\le1-\sqrt{\eta}.
\]
Let $E=\clB(x,R)\sm \dom$.
Then the maximum principle yields
\[
\hm{\bd B(x,2R)}{B(x,2R)\sm E}
\le \hm{\dom\cap\bd B(x,R)}{\dom\cap B(x,R)}
\qtext{in }\dom\cap B(x,R),
\]
so that
\[
\hm{\bd B(x,2R)}{B(x,2R)\sm E}
\le1-\sqrt{\eta}
\qtext{in }\clB(x,R-2k\rho),
\]
where we use the convention $\hm{\bd B(x,2R)}{B(x,2R)\sm E}=0$ in $E$.
Hence, \reflem{1-capa=hm} (i)
with $R-2k\rho$ and $2R$ in place of $r$ and $R$ gives
\[
1- \frac{\capain{B(x,2R)}(E)}{\capain{B(x,2R)}(\clB(x,R-2k\rho))}
\le 1-\sqrt{\eta},
\]
so that
\[
\frac{\capain{B(x,2R)}(E)}{\capain{B(x,2R)}(\clB(x,R-2k\rho))}
\ge \sqrt{\eta}.
\]
Multiplying the inequality and \refeq{cap1-Ninv}, we obtain
\[
\frac{\capain{B(x,2R)}(E)}{\capain{B(x,2R)}(\clB(x,R))}
\ge \eta,
\]
as $R-2k\rho=(1-N\inv)R$.
Since $x\in \dom$ is arbitrary,
we have $\cw(\dom)<R=2Nk\rho=2Nk w_{\eta'}(\dom)$.
Thus we have \refeq{cw<2Nkcw'}.
\end{proof}

\begin{proof}[Proof of \refthm{vd=cw2}]
First, let us prove the second inequality of \refeq{cw<G<cw}, i.e.,
$\normi{v_\dom}\le\const\cw(\dom)^2$.
In view of the monotonicity of the torsion function,
we may assume that $\dom$ is bounded and hence $\normi{v_\dom}<\infty$.
By definition we find $r$, $\cw(\dom)\le r<2\cw(\dom)<2R_0$, such that
\[
\frac{\capain{B(x,2r)}(\clB(x,r)\sm\dom)}{\capain{B(x,2r)}(\clB(x,r))}
\ge\eta
\qtext{for every }x\in\dom.
\]
For a moment we fix $x\in \dom$ and let $B=B(x,r)$,
$\BB=B(x,2r)$, and $E=\clB\sm\dom$ for simplicity.
Then $\capain\BB(E)/\capain\BB(\clB)\ge\eta$.
We compare $v_\dom$ with
\[
v_\BB=\int_{\BB}G_{\BB}(\cdot,y)d\mu(y).
\]
It is easy to see that
$v_\dom-v_\BB$ is harmonic in $\dom\cap\BB$
and $v_\dom=0$ on $\bdy$ outside a polar set.
Hence the maximum principle yields
\[
v_\dom-v_\BB\le \normi{v_\dom}\hm{\dom\cap\bd\BB}{\dom\cap\BB}
\qtext{in } \dom\cap\BB.
\]
Since \reflem{1-capa=hm} implies that
\[
\hm[x]{\dom\cap\bd\BB}{\dom\cap\BB}
\le \hm[x]{\bd\BB}{\BB\sm E}
\le 1-\refc{1-capa=hm}^{-1}\eta,
\]
it follows from \reflem{vB} that
\[
\begin{split}
v_\dom(x)
\le v_\BB(x)+ \normi{v_\dom}\hm[x]{\dom\cap\bd\BB}{\dom\cap\BB}
\le \const r^2+ \normi{v_\dom}(1-\refc{1-capa=hm}^{-1}\eta).
\end{split}
\]
Taking the supremum with respect to $x\in\dom$, we obtain
\[
\normi{v_\dom}\le\const\refc{1-capa=hm}\eta\inv r^2
\le 4\const\refc{1-capa=hm}\eta\inv \cw(\dom)^2.
\]

Second, let us prove the first inequality of \refeq{cw<G<cw}, i.e.
$\cw(\dom)^2\le\const\normi{v_\dom}$.
We distinguish two cases.
Suppose first $\normi{v_\dom}\ge\refc{vB>r2}R_0^2/2$ with $\refc{vB>r2}$ as in \refeq{vB>r2}.
Then
\[
\normi{v_\dom}\ge\refc{vB>r2}R_0^2/2>\refc{vB>r2}\cw(\dom)^2/2,
\]
as required.
Suppose next $\normi{v_\dom}<\refc{vB>r2}R_0^2/2$.
Take $R$ such that
\begin{equation}\label{eq:r2=vd}
\normi{v_\dom}
=\frac{\refc{vB>r2}R^2}2.
\end{equation}
Then $0<R<R_0$.
Let $x\in\dom$.
This time,  we let $B=B(x,R)$, $\BB=B(x,2R)$ and $E=\clB\sm\dom$
with $R$ as in \refeq{r2=vd}.
We shall compare $v_\dom$ with the torsion function
\[
v_B=\int_{B} G_{B}(\cdot,y)d\mu(y).
\]
Observe that $v_B-v_\dom$ is harmonic in $\dom\cap B$.
By the maximum principle and \reflem{vB}
\[
\begin{split}
v_B-v_\dom
&\le\sup_E v_B\cdot\hm{\bd E}{B\sm E}
=\sup_E v_B\cdot(1-\hm{\dom\cap\bd B}{B\sm E})\\
&\le \const R^2(1-\hm{\bd\BB}{\BB\sm E})
\qtext{in }\dom\cap B,
\end{split}
\]
since
\(
\bd(\dom\cap B)
\subset(B\cap\bdy)\cup(\dom\cap\bd B)
\subset E\cup\bd B,
\)
and since  $v_B=0$ on $\bd B$.
Let $0<\ve<1$ be as in \refeq{vB>r2}.  Taking the infimum over $\clB(x,\ve R)$,
we obtain from \reflem{1-capa=hm} that
\[
\inf_{\clB(x,\ve R)}v_B-\normi{v_\dom}
\le \const R^2\Big(1-\sup_{\clB(x,\ve R)}\hm{\bd\BB}{\BB\sm E}\Big)
\le \const R^2\frac{\capain{\BB}(E)}{\capain{\BB}(\clB(x,\ve R))}.
\]
Hence, \refeq{vB>r2} and \refeq{r2=vd} yield
\[
\refc{vB>r2} R^2-\frac{\refc{vB>r2}R^2}2
\le \const R^2\frac{\capain{\BB}(E)}{\capain{\BB}(\clB(x,\ve R))}.
\]
Dividing by $\const R^2$, we obtain
\[
\frac{\capain{\BB}(E)}{\capain{\BB}(\clB(x,\ve R))}
\ge\frac{\refc{vB>r2}}{2\const},
\]
so that, by \reflem{v<c} and volume doubling
\[
\frac{\capain{\BB}(E)}{\capain{\BB}(\clB(x,R))}
=\frac{\capain{\BB}(E)}{\capain{\BB}(\clB(x,\ve R))}\cdot
 \frac{\capain{\BB}(\clB(x,\ve R))}{\capain{\BB}(\clB(x,R))}
\ge \frac{\refc{vB>r2}}{2\const}\cdot\frac{\const\mu(\clB(x,\ve R))}{\mu(\clB(x,R))}
\ge\eta'
\]
with $0<\eta'<1$ depending only on $\sqrt K\,R_0$ and $\di$.
Thus
\[
\frac{\capain{\BB}(\clB(x,R)\sm\dom)}{\capain{\BB}(\clB(x,R))}\ge\eta'.
\]
Since $x\in\dom$ is arbitrary,
we have  $w_{\eta'}(\dom)<R$
and so $\cw(\dom)\le\const R$ by \refthm{cwcomp}.
Hence $\cw(\dom)^2\le\const\normi{v_\dom}$ by \refeq{r2=vd}.
The proof is complete.
\end{proof}

\section{Proof of \refthm{IU}}

The crucial step of the proof of \refthm{IU} is the following
\emph{parabolic box argument} (cf. \cite[Lemma 4.1]{MR3420485}),

\begin{lem}\label{lem:P<Ag}
Suppose \refeq{IUint} holds.
If $t>0$, then
\begin{equation}\label{eq:P<Ag}
P_\dom(t,x)\le\const_t G_\dom(x,o)
\qtext{for }x\in\dom
\end{equation}
with $\const_t$ depending on $t$.
\end{lem}

\begin{proof}
Without loss of generality we may assume that $\tau=1$ in \refeq{IUint}.
For notational convenience we shall prove \refeq{P<Ag} with $T$ in place of $t$.
For simplicity we write $\cw(\Gdo<s)=\cw(\{x\in\dom: G_\dom(x,o)<s\}$.
Let $\alpha_j=\exp(-2^j)$.
Since
\[
\begin{split}
\int_{\alpha_j}^{\alpha_{j-1}}\cw(\Gdo<s)^2\frac{ds}s
&\ge \cw(\Gdo<\alpha_j)^2 \int_{\alpha_j}^{\alpha_{j-1}}\frac{ds}s\\
&=\cw(\Gdo<\alpha_j)^2(2^j-2^{j-1})
=2^{j-1}\cw(\Gdo<\alpha_j)^2,
\end{split}
\]
it follows from \refeq{IUint} that
$\sum_{j=0}^\infty 2^{j}\cw(\Gdo<\alpha_j)^2<\infty$.

Let $\cw(\Gdo<1)<R_0<\infty$ and
choose $\refc{ke1}$ and $\refc{ke2}$  as in \refthm{P<exp(A/w2)}.
We find $j_0\ge0$ such that
\begin{equation}\label{eq:<T}
\frac{3}{\refc{ke2}}\sum_{j=j_0+1}^\infty 2^{j}\cw(\Gdo<\alpha_j)^2<T.
\end{equation}
Define
\[
t_k
=\frac{3}{\refc{ke2}}\sum_{j=j_0+1}^{k} 2^j\cw(\Gdo<\alpha_j)^2
\qtext{for $k\ge j_0+1$,}
\]
and $t_{j_0}=0$.
Then $t_k$ increases and $\lim_{k\to\infty}t_k<T$ by \refeq{<T}.
Observe that
\begin{equation}\label{eq:=e-2k}
\frac1{\alpha_{k+1}}\exp\Big(-\frac{\refc{ke2}(t_k-t_{k-1})}{\cw(\Gdo<\alpha_k)^2}\Big)
=\exp(2^{k+1}-3\cdot2^k)
=\exp(-2^k)
\end{equation}
for $k\ge j_0+1$.

Let
$\dom_k=\{x\in\dom: G_\dom(x,o)<\alpha_k\}$,
$E_k=\{x\in\dom: \alpha_{k+1}\le G_\dom(x,o)<\alpha_k\}$,
$\ddom_k=(t_{k-1},\infty)\times \dom_k$ and
$\EE_k=(t_{k},\infty) \times E_k$.
Put
\[
q_k=\sup_{(t,x)\in\EE_k}\frac{P_\dom(t,x)}{G_\dom(x,o)}.
\]
We claim that $\sup_{k\ge j_0+1}q_k\le\Const$, which implies \refeq{P<Ag}
with $T$ in place of $t$, and $\const_T=\max\{\Const,1/\alpha_{j_0+1}\}$
since $(T,\infty)\times \{x\in\dom: G_\dom(x,o)<\alpha_{j_0+1}\}\subset \union_{k\ge j_0+1}\EE_k$
by \refeq{<T}.
See \reffig{pba}.
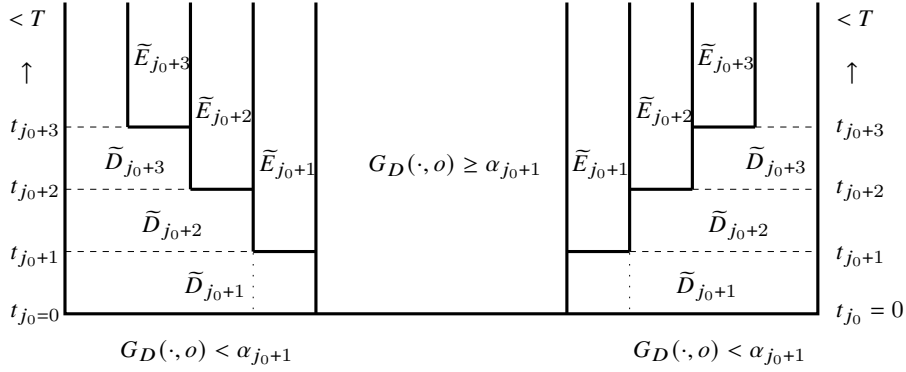
\begin{figure}[htb]
\begin{center}
\scriptsize
\begin{overpic}[scale=.65]
{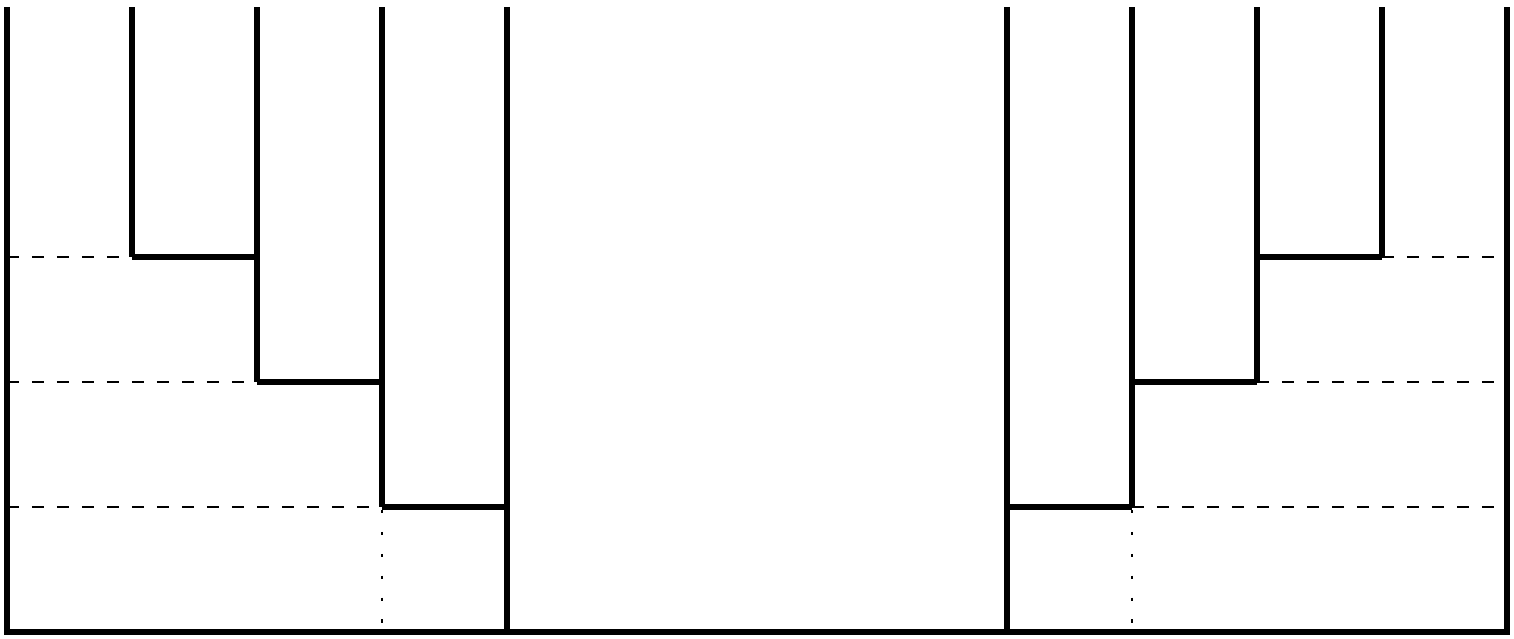}
\put(290,110){$<T$}
\put(294,90){$\increase$}
\put(290,70){$t_{j_0+3}$}
\put(290,47){$t_{j_0+2}$}
\put(290,22){$t_{j_0+1}$}
\put(290,0){$t_{j_0}=0$}

\put(-20,110){$<T$}
\put(-15,90){$\increase$}
\put(-19,70){$t_{j_0+3}$}
\put(-19,47){$t_{j_0+2}$}
\put(-19,22){$t_{j_0+1}$}
\put(-19,0){$t_{j_0=0}$}

\put(22,-15){$G_\dom(\cdot,o)<\alpha_{j_0+1}$}
\put(214,-15){$G_\dom(\cdot,o)<\alpha_{j_0+1}$}

\put(46,9){$\ddom_{j_0+1}$}
\put(230,9){$\ddom_{j_0+1}$}

\put(30,32){$\ddom_{j_0+2}$}
\put(242,32){$\ddom_{j_0+2}$}

\put(16,56){$\ddom_{j_0+3}$}
\put(256,56){$\ddom_{j_0+3}$}

\put(191,55){$\EE_{j_0+1}$}
\put(215,75){$\EE_{j_0+2}$}
\put(238,95){$\EE_{j_0+3}$}

\put(74,55){$\EE_{j_0+1}$}
\put(50,75){$\EE_{j_0+2}$}
\put(27,95){$\EE_{j_0+3}$}

\put(115,55){$G_\dom(\cdot,o)\ge\alpha_{j_0+1}$}
\end{overpic}
\vrule width0pt depth14pt 
\end{center}
\caption{Parabolic box argument.}\label{fig:pba}
\end{figure}

By the parabolic comparison principle over $\ddom_{j_0+1}$ we have
\[
P_\dom(t,x)\le\frac{G_\dom(x,o)}{\alpha_{j_0+1}}+P_{\dom_{j_0+1}}(t,x)
\qtext{for }(t,x)\in\ddom_{j_0+1}=(0,\infty)\times\dom_{j_0+1}.
\]
Divide the both sides by $G_\dom(x,o)$ and take the supremum over $\EE_{j_0+1}$.
Then \refeq{ke} and \refeq{=e-2k} yield
\[
\begin{split}
q_{j_0+1}
&\le\frac1{\alpha_{j_0+1}}+\sup_{(t,x)\in\EE_{j_0+1}}\frac{P_{\dom_{j_0+1}}(t,x)}{G_\dom(x,o)}
\le \frac1{\alpha_{j_0+1}}+\frac{\refc{ke1}}{\alpha_{j_0+2}}
   \sup_{t\ge t_{j_0+1}}\exp\Big(-\frac{\refc{ke2}t}{\cw(\dom_{j_0+1})^2}\Big)\\
&\le \frac1{\alpha_{j_0+1}}+\frac{\refc{ke1}}{\alpha_{j_0+2}}
   \exp\Big(-\frac{\refc{ke2}(t_{j_0+1}-t_{j_0+1})}{\cw(\dom_{j_0+1})^2}\Big)
=\exp(2^{j_0+1})+\refc{ke1}\exp(-2^{j_0+1}).
\end{split}
\]
Let $k\ge j_0+2$.
By the parabolic comparison principle over $\ddom_k$ we have
\[
P_\dom(t,x)\le q_{k-1} G_\dom(x,o)+P_{\dom_k}(t-t_{k-1},x)
\qtext{for }(t,x)\in\ddom_k=(t_{k-1},\infty)\times\dom_k.
\]
Divide the both sides by $G_\dom(x,o)$ and take the supremum over $\EE_k$.
In the same way as above, we obtain from \refeq{ke} and \refeq{=e-2k} that
\[
q_k
\le q_{k-1}+
\frac{\refc{ke1}}{\alpha_{k+1}}
\exp\Big(-\frac{\refc{ke2}(t_{k}-t_{k-1})}{\cw(\dom_k)^2}\Big)
\le q_{k-1}+\refc{ke1}\exp(-2^k).
\]
Hence we have the claim as
\[
\sup_{k\ge j_0+1}q_k\le\exp(2^{j_0+1})+\refc{ke1}\sum_{k=j_0+1}^\infty \exp(-2^k)<\infty.
\]
The lemma is proved.
\end{proof}

\begin{proof}[Proof of \refthm{IU}]
By \refthm{nes} we have the first condition for IU.
Let us show \refeq{IU} for every $t>0$.
It is known that the lower estimate of \refeq{IU} follows from the upper estimate.
Moreover, if $p_\dom(t_0,x,y)\le C_{t_0}\feigenf (x)\feigenf (y)$ for all $x,y\in\dom$
with some $t_0>0$, then
$p_\dom(t,x,y)\le C_t\feigenf (x)\feigenf (y)$ holds
with $C_t\le C_{t_0} e^{-\feigen (t-t_0)}$ for $t\ge t_0$
(See e.g. \cite[Proposition 2.1]{MR3420485}).
Hence, it suffices to show the upper estimate of \refeq{IU} for small $t>0$.

Since $\feigenf$ is superharmonic, and since $G_\dom(\cdot,o)$ is harmonic outside $\{o\}$,
we have $G_\dom(\cdot,o)\le\const \feigenf$ apart from a neighborhood of $o$.
So, it is sufficient to show that if $t>0$ small, then there exists $\const_t>0$ such that
\begin{equation}\label{eq:<Gdo}
p_\dom(t,x,y)\le\const_t G_\dom(x,o)G_\dom(y,o)
\qtext{for }x,y\in \dom.
\end{equation}

Let $i_0$ be the injectivity radius of $M$.
It is known that
\[
\mu(B(x,r))\ge \const r^\di
\qtext{for $0<r<i_0/2$ and $x\in M$.}
\]
where $\const>0$ depends only on $M$ (Croke \cite[Proposition 14]{MR608287}).
Hence, the Gaussian estimate \refeq{GER} yields
\begin{equation}\label{eq:pM<t-n/2}
p_M(t,x,y)\le\frac\const{V(x,\sqrt t)}\le\const t^{-\di/2}
\end{equation}
for $0<t<\min\{R_0^2,(i_0/2)^2\}$ and $x,y\in M$.
Let $0<t<\min\{R_0^2,(i_0/2)^2\}$ and $x,y,z\in\dom$.
By \refeq{pM<t-n/2} we have
\[
\begin{split}
p_\dom(2t,z,y)
&=\int_\dom p_\dom(t,z,w)p_\dom(t,w,y)d\mu(w)
\le\int_\dom p_M(t,z,w)p_\dom(t,w,y)d\mu(w)\\
&\le\const t^{-\di/2}\int_\dom p_\dom(t,w,y)d\mu(w)
=\const t^{-\di/2} P_\dom(t,y),
\end{split}
\]
since the heat kernel is symmetric.
Moreover,
\[
\begin{split}
p_\dom(3t,x,y)
&\le\int_\dom p_\dom(t,x,z)p_\dom(2t,z,y)d\mu(z)
\le\int_\dom p_\dom(t,x,z)\const t^{-\di/2} P_\dom(t,y) d\mu(z)\\
&
=\const t^{-\di/2} P_\dom(t,x)P_\dom(t,y).
\end{split}
\]
Hence \reflem{P<Ag} yields
\[
p_\dom(3t,x,y)
\le\const t^{-\di/2} P_\dom(t,x)P_\dom(t,y)
\le\const_t G_\dom(x,o)G_\dom(y,o).
\]
Replacing $3t$ by $t$, we obtain \refeq{<Gdo} for small $t>0$.
Thus the theorem is proved.
\end{proof}

\begin{rem}\label{rem:vol>}
The assumption on the injectivity radius can be replaced  by
\begin{equation}\label{eq:infvol}
\inf_{x\in M} \mu(B(x,R_0))>0.
\end{equation}
In fact,
\refeq{VDa} yields
\[
\mu(B(x,r))\ge\const \Big(\frac r{R_0}\Big)^\alpha\inf_{x\in M} \mu(B(x,R_0))
\qtext{for all $x\in M$ and  $0<r<R_0$,}
\]
and hence for small $t>0$,
\[
p_M(t,x,y)\le\frac\const{V(x,\sqrt t)}\le\const t^{-\alpha/2}.
\]
Replacing \refeq{pM<t-n/2} by this inequality, we obtain
\[
p_\dom(3t,x,y)
\le\const t^{-\alpha/2} P_\dom(t,x)P_\dom(t,y)
\le\const_t G_\dom(x,o)G_\dom(y,o),
\]
which proves \refthm{IU}.
See \cite{MR961611} for further discussion on \refeq{infvol}.
\end{rem}

\section{Remarks}\label{sec:rem}

Once we obtain the theorems in Section \ref{sec:main},
we can extend many Euclidean results to the setting of manifolds.
Proofs are almost the same as in the Euclidean case.
For instance, we relax the requirement of inner uniformity for IU
assumed in \cite[Theorem 7.9]{MR3170207}.
For a curve $\gamma$ in $M$ we denote the length of $\gamma$ and
the subarc of $\gamma$ between $x$ and $y$
by $\ell(\gamma)$ and $\gamma(x,y)$, respectively.
For a domain $\dom$ in $M$ we define the inner metric in $\dom$ as
\[
\inmet(x,y)=\inf\{\ell(\gamma): \text{ $\gamma$ is a curve connecting $x$ and $y$ in $\dom$}\}.
\]

\begin{defn}\label{defn:inuni+john}
Let $\dom$ be a domain in $M$ and let $\dis(x)=\dist(x,M\sm\dom)$.

(i)
We say that $\dom$ is a John domain if there exist $o\in\dom$ and $\const\ge1$ such that
every $x\in\dom$ is connected to $o$ by a rectifiable curve $\gamma\subset\dom$ with the property
\[
\ell(\gamma(x,z))\le \const\dis(z)\qtext{for all }z\in\gamma.
\]

(ii)
We say that $\dom$ is an inner uniform domain if there exists $\const\ge1$ such that
every pair of points $x,y\in\dom$ can be connected by a rectifiable curve $\gamma\subset\dom$
with the properties $\ell(\gamma)\le\const\inmet(x,y)$ and
\[
\min\{\ell(\gamma(x,z),\ell(\gamma(z,y)\}
\le \const\dis(z)\qtext{for all }z\in\gamma.
\]
\end{defn}

If we replace $\inmet(x,y)$ by the ordinary metric $d(x,y)$ in (ii), then we obtain a uniform domain.
By definition a John domain is necessarily bounded.
We have the following inclusions for these classes of bounded domains:
\[
\text{uniform}\subsetneqq
\text{inner uniform}\subsetneqq
\text{John}.
\]

\reffig{jon-iu} depicts a John domain that is not inner uniform.
We find a curve connecting $x$ and $o$ with the property of
\refdefn{inuni+john} (i);
yet there is no curve connecting $x$ and $y$ with the properties
of \refdefn{inuni+john} (ii)
if the gaps on the vertical segment shrink sufficiently fast.

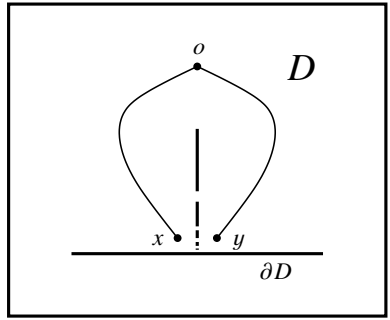
\begin{figure}[htb]
\begin{center}
\scriptsize
\begin{overpic}[height=120pt]
{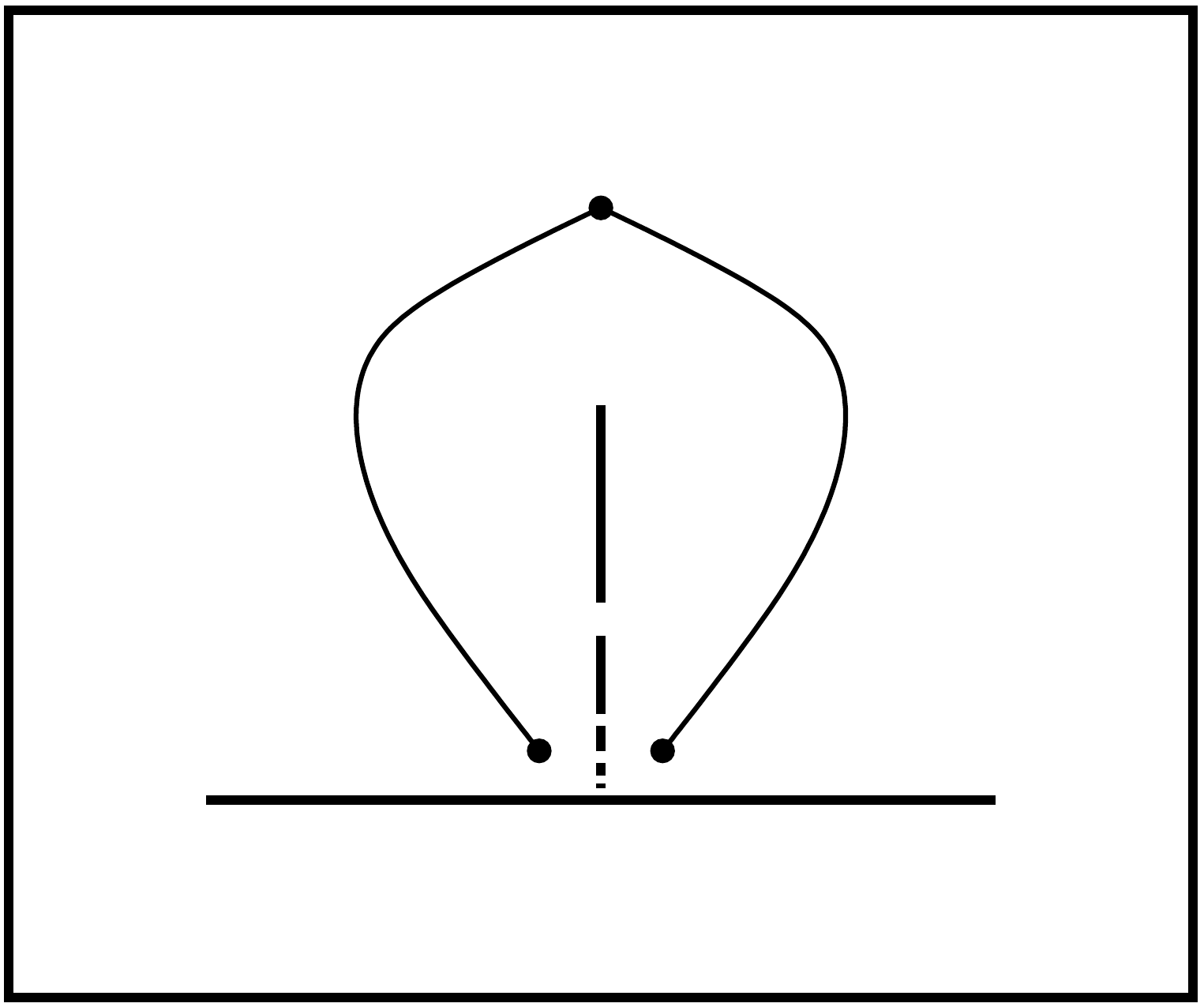}
\put(105,90){\large$\dom$}
\put(70,100){$o$}
\put(55,28){$x$}
\put(85,28){$y$}
\put(95,15){$\bdy$}
\end{overpic}
\end{center}
\caption{A John domain that is not inner uniform.}\label{fig:jon-iu}
\end{figure}

\begin{thm}\label{thm:John=IU}
A John domain is IU.
\end{thm}
\begin{proof}
Let $\dom$ be a John domain.
Observe that $\cw(\{x\in\dom:\dis(x)<r\})\le\const r$ for small $r>0$ by definition
and $G_\dom(x,o)\ge\const\dis(x)^\alpha$ with some $\alpha>0$ by the Harnack inequality.
Hence
\[
\cw(\{x\in\dom: G_\dom(x,o)<t\})
\le \cw(\{x\in\dom: \dis(x)<(t/\const)^{1/\alpha}\})
\le \const t^{1/\alpha},
\]
so that \refeq{IUint} holds.
Therefore \refthm{IU} asserts that $\dom$ is IU.
\end{proof}

\def\cprime{$'$} \def\cprime{$'$} \def\cprime{$'$} \def\cprime{$'$}
\providecommand{\bysame}{\leavevmode\hbox to3em{\hrulefill}\thinspace}
\providecommand{\MR}{\relax\ifhmode\unskip\space\fi MR }
\providecommand{\MRhref}[2]{%
  \href{http://www.ams.org/mathscinet-getitem?mr=#1}{#2}
}
\providecommand{\href}[2]{#2}

\end{document}